\documentclass[10pt]{amsart}
\usepackage{amssymb}
\usepackage{etex}
\usepackage[all]{xy}
\usepackage{subfigure} 
\usepackage{xmpmulti}  
\usepackage{colortbl,dcolumn}     
\usepackage[all]{xy}

\newtheorem{proposition}{Proposition}[section]
\newtheorem{lemma}[proposition]{Lemma}

\newtheorem{theorem}[proposition]{Theorem}

\theoremstyle{definition}
\newtheorem{definition}[proposition]{Definition}

\numberwithin{equation}{section}

\newcommand{\selabel}[1]{\label{se:#1}}

\def\<{\leqslant}
\def\>{\geqslant}
\def\a{\alpha}
\def\b{\beta}
\def\d{\delta}

\def\ol{\overline}

\def\t{\triangle}
\def\e{\varepsilon}

\def\la{\lambda}
\def\La{\Lambda}

\def\ot{\otimes}
\def\ra{\rightarrow}

\date{}

\begin{document}
\title{ The structures of Hopf $\ast$-algebra on Radford algebras}
\author{Hassan Suleman Esmael MOHAMMED}
\address{ School of Mathematical Science, Yangzhou University,
Yangzhou 225002, China}
\email{esmailhassan313@yahoo.com}
\author{Hui-Xiang Chen}
\address{School of Mathematical Science, Yangzhou University,
Yangzhou 225002, China}
\email{hxchen@yzu.edu.cn}

\begin{abstract}
We investigate the structures of Hopf $\ast$-algebra on the Radford algebras over $\mathbb {C}$.
All the $*$-structures on $H$ are explicitly given. Moreover, these Hopf $*$-algebra structures
are classified up to equivalence.
\end{abstract}
\maketitle
\section{\bf Introduction}\selabel{1}
Woronowicz studied compact matrix pseudogroup in \cite{Woronowicz1},
which is a generalization of compact matrix group. Using the language of $C^*$-algebra,
Woronowicz described compact matrix pseudogroups as $C^*$-algebras endowed with
some comultiplications. This induces the concept of Hopf $*$-algebras.
In \cite{Woronowicz1, Woronowicz2, Woronowicz3}, Woronowicz exhibited Hopf $*$-algebra structures on quantum
groups in the framework of $C^*$-algebras. It was shown that $GL_q(2)$, $SL_q(2)$ and
$U_q(sl(2))$ are Hopf $*$-algebras, see \cite{Ka, MMNNSU}.
Van Deale \cite{VD} studied the Harr measure on a compact quantum group.
Podle\'{s} \cite{PO} studied coquasitriangular Hopf $\ast$-algebras.
Tucker-Simmons \cite{TM} studied the $*$-structure of module algebras over a Hopf $\ast$-algebra.
Recently, we investigated the Hopf $\ast$-algebra structures on $H(1,q)$ over ${\mathbb C}$
and classified these $*$-structures up to the equivalence \cite {MLiChen}.

Radford \cite{Rad1} constructed  for every integer $n>1$ a finite dimensional unimodular Hopf algebra
with antipode of order 2n and proved that for every even integer there is a finite dimensional Hopf algebra $H$.
For more details, the reader is directed to \cite{Loren, Rad1, Rad2}.

In this paper, we study the structures of Hopf $\ast$-algebra on the Radford algebra $H$ over
the complex number field ${\mathbb C}$. This paper is organized as follow.
In Section 2, we recall some basic notions about the Hopf $\ast$-algebra,
and the Radford algebra $H$. In Section 3, we first describe all structures of Hopf $\ast$-algebra on Radford algebra.
It is shown that when $n>2$, a Hopf $*$-algebra structure on $H$ is uniquely determined by a pair $(\a, \b)$
of elements in $\mathbb C$ with $|\a|=|\b|=1$, and that when $n=2$, a Hopf $*$-algebra structure on $H$
is uniquely determined by a $2\times 2$-matrix $A$ over $\mathbb C$ with $\ol{A}A=I_2$.
Then we classify the Hopf $*$-algebra structures up to equivalence.
It is shown that any two $*$-structures on $H$ are equivalent when $n>2$.
When $n=2$, the two $*$-structures determined by two matrices $A$ and $B$, respectively,
are equivalent if and only if there exists an invertible $2\times 2$-matrix $\La$ over $\mathbb C$
such that $A\ol{\La}=\La B$.

\section{\bf Preliminaries}\selabel{2}

Throughout, let $\mathbb {Z}$, $\mathbb {N}$, $\mathbb R$ and $\mathbb C$ denotes the all integers,
all nonnegative integers, the field of real numbers, and the field of complex numbers, respectively.
Let $i\in \mathbb C$ be the imaginary unit.
For any $\la\in\mathbb{C}$, let $\ol{\la}$ denote the conjugate complex number of $\la$,
and let $|\la|$ denote the norm of $\la$. For a Hopf algebra $H$, we use $\triangle$, $\varepsilon$ and $S$
denote the comultiplication, the counit, and the antipode of $H$ as usual.
For the theory of quantum groups and Hopf algebras, we refer to \cite{Ka, Maj, Mon, Rad2, SM}.
Let $G(H)$ denote the set of group-like elements in a Hopf algebra $H$, which is a group.

Let $V$ and $W$ be vector spaces over $\mathbb C$.
A mapping $\psi: V\rightarrow W$ is said to be conjugate-linear (or antilinear) if
$$\psi(\la_{1} v_1+\la_{2} v_2)=\ol{\la_{1}}\psi(v_1)+\ol{\la_{2}}\psi(v_2),
\ \forall v_1, v_2\in V, \ \forall\la_1, \la_2\in\mathbb{C}.$$
Let $A$ and $B$ be $\mathbb C$-algebras.
A conjugate-linear map $\psi:{A}\rightarrow {B}$ (resp., $A$) is said to be a conjugate-linear algebra map
(resp., a conjugate-linear algebra endomorphism) if
$$\psi(aa')=\psi(a)\psi(a'), \ \psi(1)=1, \ \forall {a},{a'}\in A,$$
and $\psi$ is said to be a conjugate-linear antialgebra  map (resp., a conjugate-linear antialgebra endomorphism) if
$$\psi(aa')=\psi(a')\psi(a), \ \psi(1)=1, \ \forall {a},{a'}\in A.$$

Let $C$ and $D$ be two coalgebras over $\mathbb C$.
A conjugate-linear map $\psi: C\rightarrow D$ (resp., $C$) is said to be a conjugate-linear coalgebra map
(resp., a conjugate-linear coalgebra endomorphism) if
$$\sum\psi(c)_1\ot \psi(c)_2=\sum \psi(c_1)\ot \psi(c_2), \ \e(\psi(c))=\ol{\e(c)}, \ \forall c\in C,$$
and $\psi$ is said to be a conjugate-linear anticoalgebra map (resp., a conjugate-linear anticoalgebra endomorphism) if
$$\sum \psi(c)_1\ot \psi(c)_2=\sum \psi(c_2)\ot \psi(c_1), \ \e(\psi(c))=\ol{\e(c)}, \ \forall c\in C.$$

\begin{definition}\label{2.1}
Let $H$ be a Hopf algebra over $\mathbb C$. A $\ast$-structure on $H$ is a conjugate-linear map $\ast: H\ra H$
such that the following conditions are satisfied:
$$\begin{array}{ll}
(h^\ast)^\ast=h,& (hl)^\ast={l}^\ast {h}^\ast,\\
\sum{(h^\ast)_1\ot (h^\ast)_2}=\sum{(h_1)^\ast \ot (h_2)^\ast},
&S(S(h^*)^*)=h,\\
\end{array}$$
where $h,l\in H$. If $H$ is equipped with a $\ast$-structure,
then we call $H$ a Hopf $\ast$-algebra.
Two $\ast$-structures ${\ast'}$ and ${\ast''}$ on $H$ are said to be equivalent
if there exists a Hopf algebra automorphism $\psi$ of $H$ such that
$\psi(h^{*'})=\psi(h)^{*''}$ for all $h\in H$.
\end{definition}

Let $H$ be a Hopf $\ast$-algebra. Then it is not difficult to check that
$${\e(h^{\ast})}=\ol{\e(h)}, \ \forall h{\in H}.$$
Hence, the map $\ast$ is an antilinear coalgebra endomorphism of $H$.
$\mathbb{C}=\mathbb{C}1_H$ is a subalgebra of $H$. In this case, $\la^*=\ol{\la}$
for any $\la\in\mathbb C\subseteq H$.

Fix a positive integer $n>1$ and let $\omega\in\mathbb C$ be a root of unity of order $n$.
The Radford algebra $H$ over $\mathbb C$ is generated, as a $\mathbb C$-algebra, by $g$, $x$ and $y$
subject to the relations:
$$\begin{array}{lllll}
g^{n}=1,& x^{n}=y^{n}=0, & xg=\omega gx,& gy=\omega yg,& xy=\omega yx.\\
\end{array}$$
Then $H$ is a Hopf algebra with the coalgebra structure and the antipode given by
$$\begin{array}{lll}
\t(g)=g\otimes g, & \e(g)=1, & S(g)=g^{n-1},\\
\t(x)=x\otimes g+1\otimes x, & \e(x)=0, & S(x)=-xg^{n-1},\\
\t(y)=y\otimes g+1\otimes y,& \e(y)=0, & S(y)=-yg^{n-1}.
\end{array}
$$
Note that $H$ has a canonical basis $\{g^lx^ry^s|0\< l, r, s<n\}$ over $\mathbb C$.
For the details, the reader is directed to \cite{Loren, Rad1, Rad2}.

\section{\bf  The structres of Hopf $\ast$-algebras on $H$}\label{4}

Throughout this section, let $H$ be the Radford algebra over $\mathbb C$ described in the last section.
In this section, we study the $*$-structures on the Hopf algebra $H$.
Let $Z(H)$ denote the center of $H$. Note that $H$ is generated, as an algebra over $\mathbb R$, by ${g,x,y}$,
and $i$ subject to the relations given in the last section together with ${i^2}=-1$ and $i\in Z(H)$.
In the following, let ${H^{op}}$ denote the opposite algebra of $H$. For any $h, l\in H^{op}$,
let $h\cdot l$ denote the product of $h$ and $l$ in $H^{op}$, i.e., $h\cdot l=lh$.

\begin{lemma}\label{3.1}
Let $\a, \b\in\mathbb{C}$ with $|\a|=|\b|=1$. Then
$H$ is a Hopf $\ast$-algebra with the $\ast$-structure given by
$$g^*=g, \ x^*=\a x, \ y^*=\b y,$$
\end{lemma}

\begin{proof} We first prove that the relations given in the lemma together with ${i^\ast}={-i}$
give rise to a real antialgebra endomorphism of $H$, i.e., a real algebra map from $H$ to $H^{op}$.
Since $|\omega|$=$1$, we have $\omega^\ast=\ol{\omega}=\omega^{-1}$.
Hence in $H^{op}$, we have
$(g^*)^n=g^n=1$,
$(x^\ast)^n= (\a x)^n=0$,
$x^\ast\cdot g^\ast=\a x\cdot g=\a gx=\a \omega^{-1}xg=\a\omega^{-1}g\cdot x=\omega^\ast g^\ast\cdot x^\ast$
and $x^\ast\cdot y^\ast=\a\b x\cdot y=\a\b yx=\a\b\omega^{-1}xy=\a\b\omega^{-1}y \cdot x=\omega^\ast y^\ast\cdot x^\ast$.
Similarly, one can check that
$(y^*)^n=0$
and
$g^\ast\cdot y^\ast=\omega^{\ast}y^\ast\cdot g^\ast$.
We also have $i^\ast=-i\in Z(H^{op})$ and $(i^\ast)^2=(-i)^2=-1$.
This shows that the relations given in the lemma together with ${i^\ast=-i}$ determine a real algebra map
$\ast: H\ra {H^{op}}$.
Then it follows that $\ast$  is a conjugate-linear antialgebra endomorphism of $H$.
Hence the composition $\ast\circ\ast$  is a complex algebra endomorphism of $H$.
It is not difficult to check that $(h^\ast)^\ast=h$, $\forall h\in\{g,x,y\}$,
and so $(h^\ast)^\ast=h$, $\forall h\in{H}$. Thus, $\ast$ is an involution of $H$.
Note that both $\t\circ\ast$ and  $(\ast\ot\ast)\circ \t$  are conjugate-linear antialgebra maps from
$H$ to $H\otimes H$. It is easy check that $\t(h^\ast)=\sum(h_1)^\ast\ot(h_2)^\ast$ for any $h\in\{g,x,y\}$.
It follows that $\t(h^\ast)=\sum(h_1)^\ast\ot (h_2)^\ast$ for all $h\in H$.
Similarly, we have $\e(h^*)=\overline{\e(h)}$ for all $h\in H$.
Finally, since $S$ is a complex antialgebra endomorphism of $H$ and $\ast$ is a conjugate-linear antialgebra endomorphism of $H$,
the map $H\rightarrow H$, $h\mapsto{S(S(h^\ast)^\ast)}$
is a complex algebra endomorphism of $H$. Now we have

$$\begin{array}{rl}
S(S(g^\ast)^*)&=S(S(g)^\ast))=S((g^{-1})^\ast)=S((g^{\ast})^{-1})=S(g^{-1})=g,\\
S(S(x^\ast)^\ast)&=S(S(\a x)^*)=S((-\a xg^{n-1})^\ast)=S(-\ol{\a}(g^{n-1})^\ast x^{\ast})\\
&=S(-\ol{\a}\a g^{n-1}x)=-S(x)S(g^{n-1})=xg^{n-1}g=x,\\
\end{array}$$
and similarly $S(S(y^\ast)^\ast)=y$.
It follows that $S(S(h^\ast)^\ast)=h$ for all $h\in{H}$.
\end{proof}

Let $M_2(\mathbb{C})$ be the matrix algebra of all $2\times 2$-matrices over $\mathbb C$.
For a matrix
$$A=\left(
  \begin{array}{cc}
    \a_{11} & \a_{12} \\
    \a_{21} & \a_{22}  \\
  \end{array}
\right)\in M_2(\mathbb C),$$
let
$$\ol{A}=\left(
  \begin{array}{cc}
    \ol{\a_{11}} & \ol{\a_{12}}  \\
    \ol{\a_{21}} & \ol{\a_{22}} \\
  \end{array}
\right)\in M_2(\mathbb C).$$

\begin{lemma}\label{3.2} Assume that $n=2$ and let
$A=\left(
       \begin{array}{cc}
         \a_{11} & \a_{12} \\
         \a_{21} & \a_{22} \\
       \end{array}
     \right)
\in M_2(\mathbb{C})$ with
$\ol{A}A=I_2$, the $2\times 2$ identity matrix.
Then $H$ is a Hopf $\ast$-algebra with the $\ast$-structure given by
$$g^*=g, \ x^*=\a_{11}x+\a_{12}y, \ y^*=\a_{21}x+\a_{22}y.$$
\end{lemma}

\begin{proof} Assume that $n=2$. Then $\omega=-1$.
We first prove that the relations given in the lemma together with ${i^\ast}={-i}$
give rise to a real antialgebra endomorphism of $H$, i.e., a real algebra map from $H$ to $H^{op}$.
In $H^{op}$, we have
$(g^{\ast})^2=g^2=1$,
$(x^*)^2= (\a_{11}x+\a_{12}y)^2=\a_{11}^2x^2+\a_{11}\a_{12}xy+\a_{12}\a_{11}yx+\a_{12}^2y^2=0$
and
$x^*\cdot g^*=(\a_{11}x+\a_{12}y)\cdot g
=\a_{11}gx+\a_{12}gy=-\a_{11}xg-\a_{12}yg
=-g\cdot(\a_{11}x+\a_{12}y)=- g^*\cdot x^*$.
We also have $x^*\cdot y^*=(\a_{21}x+\a_{22}y)(\a_{11}x+\a_{12}y)
=\a_{21}\a_{11}x^2+\a_{21}\a_{12}xy+\a_{22}\a_{11}yx+\a_{22}\a_{12}y^2
=(\a_{21}\a_{12}-\a_{22}\a_{11})xy$ and
$y^*\cdot x^*=(\a_{11}x+\a_{12}y)(\a_{21}x+\a_{22}y)
=\a_{11}\a_{21}x^2+\a_{11}\a_{22}xy+\a_{12}\a_{21}yx+\a_{12}\a_{22}y^2
=(\a_{11}\a_{22}-\a_{12}\a_{21})xy$, which implies that $x^*\cdot y^*=-y^*\cdot x^*$.
Similarly, one can check that
$(y^*)^2=0$
and
$g^\ast\cdot y^\ast=-y^\ast\cdot g^\ast$.
We also have $i^\ast=-i\in Z(H^{op})$ and $(i^\ast)^2=(-i)^2=-1$.
This shows that the relations given in the lemma together with ${i^\ast=-i}$ determine a real algebra map
$\ast: H\ra {H^{op}}$.
Then it follows that $\ast$  is a conjugate-linear antialgebra endomorphism of $H$.
Hence the composition $\ast\circ\ast$  is a complex algebra endomorphism of $H$.
Clearly, $(g^*)^*=g$. Since $\ol{A}A=I_2$, $\ol{\a_{i1}}\a_{1j}+\ol{\a_{i2}}\a_{2j}=\d_{ij}$
for $1\<i, j\<2$.
Hence we have
$$\begin{array}{rl}
(x^*)^*&=(\a_{11}x+\a_{12}y)^*=\ol{\a_{11}}x^*+\ol{\a_{12}}y^*\\
&=\ol{\a_{11}}(\a_{11}x+\a_{12}y)+\ol{\a_{12}}(\a_{21}x+\a_{22}y)\\
&=(\ol{\a_{11}}\a_{11}+\ol{\a_{12}}\a_{21})x+(\ol{\a_{11}}\a_{12}+\ol{\a_{12}}\a_{22})y=x.\\
\end{array}$$
Similarly, we also have $(y^*)^*=y$. It follows that $(h^*)^*=h$ for all $h\in{H}$.
Thus, $\ast$ is an involution of $H$.
Note that both $\t\circ\ast$ and  $(\ast\ot\ast)\circ \t$  are conjugate-linear antialgebra maps from
$H$ to $H\otimes H$. It is easy check that $\t(h^\ast)=\sum(h_1)^\ast\ot(h_2)^\ast$ for any $h\in\{g,x,y\}$.
It follows that $\t(h^\ast)=\sum(h_1)^\ast\ot (h_2)^\ast$ for all $h\in H$.
Similarly, we have $\e(h^*)=\overline{\e(h)}$ for all $h\in H$.
Finally, since $S$ is a complex antialgebra endomorphism of $H$ and $\ast$ is a conjugate-linear antialgebra endomorphism of $H$,
the map $H\rightarrow H$, $h\mapsto{S(S(h^\ast)^\ast)}$
is a complex algebra endomorphism of $H$. Now, we have
$$\begin{array}{rl}
S(S(g^*)^*)&=S(S(g)^*))=S(g^*)=S(g)=g,\\
S(S(x^*)^*)&=S(S(\a_{11}x+\a_{12}y)^*)=S((-\a_{11}xg-\a_{12}yg)^*)\\
&=S((\a_{11}gx+\a_{12}gy)^*)=S(\ol{\a_{11}}x^*g^*+\ol{\a_{12}}y^*g^*)\\
&=S(\ol{\a_{11}}(\a_{11}x+\a_{12}y)g+\ol{\a_{12}}(\a_{21}x+\a_{22}y)g)\\
&=S(g)S((\ol{\a_{11}}\a_{11}+\ol{\a_{12}}\a_{21})x+(\ol{\a_{11}}\a_{12}+\ol{\a_{12}}\a_{22})y)\\
&=S(g)S(x)=g(-xg)=x,\\
\end{array}$$
and similarly $S(S(y^\ast)^\ast)=y$.
It follows that $S(S(h^\ast)^\ast)=h$ for all $h\in{H}$.
\end{proof}

The following proposition follows similarly to \cite[Lemma 2.7]{Ch1}.

\begin{proposition}\label{3.3}
For any $r,s\in{\mathbb N}$ and $l\in{\mathbb Z}$,\\
$$\
\t(y^rx^sg^l)=\sum_{i=0}^{r}\sum_{j=0}^{s}\omega^{-(r-i)j}\binom{r}{i}_{\omega}\binom{s}{j}_{\omega^{-1}}
y^{r-i}x^{s-j}g^{l}\ot y^ix^{j}g^{l+s-j+r-i}.
$$
\end{proposition}

\begin{proof} Since
$$
(x\ot g)(1\ot x)=\omega^{-1}(1\ot x)(x\ot g), \ (y\ot g)(1\ot y)=\omega(1\ot y)(y\ot g),
$$
it follows from \cite[Proposition IV.2.2]{Ka} that
$$\begin{array}{ll}
\t(x)^s=(1\ot x+x\ot g)^s=\sum_{j=0}^s{\binom{s}{j}}_{\omega^{-1}}{x^{s-j}}\ot{x^{j}g^{s-j}},\\
\t(y)^r=(1\ot y+y\ot g)^r=\sum_{i=0}^{r}{\binom {r}{i}}_{\omega}{y^{r-i}}\ot{y^{i}}g^{r-i}.\\
\end{array}$$
Now, since $\t$ is an algebra map, we have
$$\begin{array}{rl}
\t({y^rx^sg^l})=&\t(y)^r\t(x)^s\t(g)^l\\
=&(1\ot y+y\ot g)^r(1\ot x+x\ot g)^s(g\ot g)^l\\
=&\sum_{i=0}^{r}\sum_{j=0}^{s}\binom{r}{i}_{\omega}\binom{s}{j}_{\omega^{-1}}y^{r-i}x^{s-j}g^l\ot y^ig^{r-i}x^{j}g^{l+s-j}\\
=&\sum_{i=0}^{r}\sum_{j=0}^{s}\omega^{-(r-i)j}\binom{r}{i}_{\omega}\binom{s}{j}_{\omega^{-1}}y^{r-i}x^{s-j}g^{l}\ot y^ix^{j}g^{l+s-j+r-i}.\\
\end{array}$$
\end{proof}

Note that $\{y^rx^sg^l|0\leqslant r,s,l<n\}$ is a canonical basis of $H$ over $\mathbb {C}$. Hence,
$$
\{y^rx^sg^l\ot y^{r_1}x^{s_1}g^{l_1}|0\leqslant r, r_1, s, s_1, l, l_1<n\}.
$$
is a basis of $H\ot H$ over $\mathbb{C}$.
For an element
$$h=\sum_{0\leqslant r,s,l<n}\la_{r,s,l}y^{r}x^{s}g^{l}$$
in $H$, if  $\la_{r,s,l}\neq 0$, then we say that ${y^rx^sg^l}$ is a term of $h$.
Moreover, $r$ (resp., $s$) is called the degree of $y$ (resp., $x$) in the term ${y^rx^sg^l}$.
Similarly, for an element
$$h=\sum_{0\leqslant r,s,l,{r_{1}},{s_{1}},{l_{1}}<n}\la_{r,s,l,{r_{1}},{s_{1}},{l_{1}}}y^{r}x^{s}g^{l}\ot y^{r_{1}}x^{s_{1}}g^{l_{1}}$$
in $H\ot H$, if $\la_{r,s,l,{r_{1}},{s_{1}},{l_{1}}}\neq  0$, then we say that  ${y^{r}x^{s}g^{l}}\ot y^{r_{1}}x^{s_{1}}g^{l_{1}}$
is a term of $h$. Moreover, ${r+{r_{1}}}$ (resp., ${s+{s_{1}}}$) is called the total degree of $y$ (resp., $x$) in the term
$${y^{r}x^{s}g^{l}}\ot {y^{r_{1}}x^{s_{1}}g^{l_{1}}}.$$

\begin{lemma}\label{3.4}
$G(H)=\{g^l|0\leqslant l<n\}$.
\end{lemma}

\begin{proof} Obviously, $g^l\in G(H)$ for all $0\leqslant l<n$.
Conversely, let $$h=\sum_{0\leqslant r,s,l<n}\la_{{r},{s},{l}}y^rx^{s}g^{l}\in G(H),$$
where ${\la_{{r},{s},{l}}\in \mathbb C}$.
Assume that ${r_{1}}$ is the highest degree of $y$ in the terms of $h$, that is, there is a nonzero coefficient
${\la_{{r_{1}},{s_{1}},{l_{1}}}\neq 0}$ in the above expression of ${h}$ such that $\la_{{r},{s},{l}}\neq 0$
implies $r\leqslant r_{1}$. From Proposition \ref {3.3}, one knows that the total degree of $y$
in each term of the expression of $\t(y^rx^sg^l)$ is $r$. Then from $$\t(h)=\sum_{r,s,l}\la_{r,s,l}\t(y^rx^sg^l),$$
one gets that the highest total degree of $y$ in the terms of $\t(h)$ is ${r_{1}}$. However,
$$y^{r_1}x^{s_1}g^{l_1}\ot y^{r_1}x^{s_1}g^{l_1}$$
is a term of  $h\ot h$  with the nonzero coefficient ${{\la^{2}}_{{r_1},{s_1},{l_1}}\neq 0}$.
It follows from $\t(h)=h\ot h$ that $0\leqslant {2r_{1}}\leqslant {r_{1}}$, which implies  ${r_{1}}=0$.
Thus, if $r>0$ then ${\la_{r,s,l}}=0$. Similarly, one can show that ${\la_{r,s,l}}=0$ for any $s>0$.
Therefore, $h\in{\rm span}\{g^l|0\leqslant l<n\}$. Since $G(H)$ is linearly independent over ${\mathbb C}$ and
$\{g^l|0\leqslant l<n\}\subseteq G(H)$, we have $h=g^l$ for some $0\leqslant l<n$. Hence
$G(H)\subseteq \{g^l|0\leqslant l<n\}$, and so $G(H)=\{g^l|0\leqslant l<n\}$.
\end{proof}

\begin{lemma}\label{4.6}
Let $h\in H$. If $\triangle(h)={h\ot g+1\ot h}$, then
$h=\la_{1} x+\la_2y+\la_{3}(1-g)$ for some $\la_1, \la_2, \la_3\in{\mathbb C}$.
\end{lemma}

\begin{proof}
Let $h=\sum_{0\leqslant r,s,l<n}\la_{r,s,l}{y^rx^sg^l}$ with $\la_{r,s,l}\in{\mathbb C}$ such that
$$\triangle(h)=h\ot g+1\ot h.$$
Then $\e(h)=0$.
For any $0\leqslant r,s<n$, let
$$h_{r,s}=\sum_{l=0}^{n-1}\la_{r,s,l}{y^rx^sg^l}.$$
Then by Proposition \ref{3.3} and the proof of Lemma \ref{3.4}, one knows that
$$\triangle(h_{r,s})=h_{r,s}\ot g+1\ot h_{r,s}, \ \forall r,s.$$
Hence, one may assume that
$$h=y^{r}x^s(\sum_{l=0}^{n-1}\la_{l}g^l)\neq 0$$
for some $\la_{l}\in{\mathbb C}$, where $r$ and $s$ are fixed integers with $0\leqslant r, s<n$.
Now, by Proposition \ref{3.3}, we have
\begin{equation}
\label{a}
\t(h)=\sum_{l=0}^{n-1}\sum_{i=0}^r\sum_{j=0}^s\la_{l}\omega^{-(r-i)j}
\binom{r}{i}_{\omega}\binom{s}{j}_{\omega^{-1}}y^{r-i}x^{s-j}g^{l}\ot y^ix^{j}g^{l+s-j+r-i}
\end{equation}
and
\begin{equation}
\label{b}
\t(h)={h\ot g+1\ot h}=\sum_{l=0}^{n-1}\la_{l}(y^rx^sg^l\ot{g+1}\ot y^rx^sg^l).
\end{equation}
By the paragraph before Lemma \ref{3.4}, $H\ot H$ has a canonical basis over ${\mathbb C}$
$$\{y^rx^sg^l\ot y^{r_1}x^{s_1}g^{l_1}|0\leqslant r, r_1, s, s_1, l, l_1<n\}.$$
Now by comparing the coefficients of the basis element $g^l\ot y^rx^sg^l$
in the two expressions of $\t(h)$ given above, one gets that ${\la_{l}=0}$
if $l>1$, and that $\la_1=0$ if $(r,s)\neq(0, 0)$. Hence, $h=\la_0y^rx^s$
when $r+s\neq 0$, and $h=\la_0+\la_1g$ when $r=s=0$.

If $h=\la_0+\la_1g$, then $\la_0+\la_1=0$ by $\e(h)=0$, and so $h=\la_0(1-g)$.
Now assume $h=\la_0y^rx^s$ with $r+s\neq0$. Then (\ref{a}) and (\ref{b}) becomes
\begin{equation}
\label{a2}
\t(h)=\sum_{i=0}^r\sum_{j=0}^s\la_0\omega^{-(r-i)j}
\binom{r}{i}_{\omega}\binom{s}{j}_{\omega^{-1}}y^{r-i}x^{s-j}\ot y^ix^{j}g^{s-j+r-i}
\end{equation}
and
\begin{equation}
\label{b2}
\t(h)={h\ot g+1\ot h}=\la_0(y^rx^s\ot g+1\ot y^rx^s),
\end{equation}
respectively. If both $r>0$ and $s>0$, then by comparing the coefficients of the basis element $y^r\ot x^sg^r$
in the two expressions of $\t(h)$ given above, one gets that $\la_0=0$, and hence $h=0$,
a contradiction. Hence either $r>0$ and $s=0$, or $r=0$ and $s>0$.
If $r>0$ and $s=0$, then $h=\la_0y^r$, and (\ref{a2}) and (\ref{b2}) becomes
\begin{equation}
\label{a3}
\t(h)=\sum_{i=0}^r\la_0\binom{r}{i}_{\omega}y^{r-i}\ot y^ig^{r-i}
\end{equation}
and
\begin{equation}
\label{b3}
\t(h)={h\ot g+1\ot h}=\la_0(y^r\ot g+1\ot y^r),
\end{equation}
respectively. If $r>1$, then by  comparing the coefficients of the basis element $y^{r-1}\ot yg^{r-1}$
in the two expressions of $\t(h)$ given above, one gets that $\la_0=0$, and hence $h=0$,
a contradiction. Hence $r=1$ and so $h=\la_0 y$.
Similarly, one can check that if $r=0$ and $s>0$ then $h=\la_0 x$.
This completes the proof.
\end{proof}

\begin{lemma}\label{4.7} Let $h\in H$ with $\t(h)=h\ot g^w+1\ot h$ for some $1<w<n$.
Then $h=\la(1-g^w)$ for some $\la\in\mathbb {C}$.
\end{lemma}

\begin{proof} It is similar to the proof of Lemma \ref{4.6}.
We only need to consider the case of
$$h=y^rx^s(\sum_{l=0}^{n-1}\la_{l}g^l)\neq 0$$
for some $\la_{l}\in\mathbb{ C}$, where $r$ and $s$ are fixed integers with $0\leqslant r, s<n$.
Then we have
\begin{equation}
\label{m}
\t(h)=h\ot g^w+1\ot h=\sum_{l=1}^{n-1}\la_{l}(y^rx^sg^l\ot g^w+1\ot y^rx^sg^l)
\end{equation}
and
\begin{equation}
\label{n}
\t(h)=\sum_{l=0}^{n-1}\sum_{i=0}^r\sum_{j=0}^s\la_{l}\omega^{-(r-i)j}
\binom{r}{i}_{\omega}\binom{s}{j}_{\omega^{-1}}y^{r-i}x^{s-j}g^{l}\ot y^ix^{j}g^{l+s-j+r-i}.
\end{equation}
If $r\neq 0$ and $s\neq 0$, then by comparing the coefficients of the basis element $y^rg^l\ot x^sg^{l+r}$
in the two expressions of $\t(h)$ given above, one gets that ${\la_{l}=0}$
for all $0\leqslant l<n$, and hence $h=0$, a contradiction. So $r=0$ or $s=0$.
Assume that $r\neq 0$. Then $s=0$. In this case, by comparing the coefficients of the basis element $g^l\ot y^rg^l$
in the two expressions of $\t(h)$ given above, one gets that ${\la_{l}=0}$
if $l>0$. Hence $h=\la_0y^r$, and (\ref{m}) and (\ref{n}) become
\begin{equation}
\label{m2}
\t(h)=h\ot g^w+1\ot h=\la_0(y^r\ot g^w+1\ot y^r)
\end{equation}
and
\begin{equation}
\label{n2}
\t(h)=\sum_{i=0}^r\la_0\binom{r}{i}_{\omega}y^{r-i}\ot y^ig^{r-i}
\end{equation}
respectively. Then by comparing the coefficients of the basis element $y^r\ot g^r$
in the both expressions of $\t(h)$ given in (\ref{m2}) and (\ref{n2}),
one gets that $r=w>1$ since $h=\la_0y^r\neq 0$.
Now by comparing the coefficients of the basis element $y^{r-1}\ot yg^{r-1}$
in the both expressions of $\t(h)$ given in (\ref{m2}) and (\ref{n2}),
one finds that $\la_0=0$, a contradiction. This shows that $r=0$.
Similarly, one can show that $s=0$. Hence $h=\sum_{l}\la_{l}g^l\neq 0$.
Then it is easy to see that $h=\la(1-g^w)$ for some $\la\in{\mathbb C}$.
\end{proof}

\begin{theorem}\label{4.8}
$(1)$ If $n>2$, then Lemma \ref{3.1} gives all Hopf $\ast$-algebra structures on $H$.\\
$(2)$ If $n=2$, then Lemma \ref{3.2} gives all Hopf $\ast$-algebra structures on $H$.
\end{theorem}

\begin{proof}
Assume that $H$ has a Hopf $\ast$-algebra structure $\ast$. Then
$$\t(g^\ast)=(*\ot *)\t(g)=g^*\ot g^*\ \text{ and }\ \e(g^*)=\overline{\e(g)}=1.$$
Hence $g^*\in G(H)$. By Lemma \ref{3.4}, $g^*=g^w$ for some $0\leqslant w<n$.
Since $*$ is an involution and $1^*=1$, $g^*\neq 1$. Hence $w\neq 0$, and so $0<w<n$.
We also have
$$\t(x^\ast)=(\ast\ot \ast)\t(x)=x^\ast\ot g^\ast+1^\ast\ot x^\ast=x^\ast\ot g^w+1\ot x^\ast.$$
If $w\neq 1$, then it follows from Lemma \ref{4.7} that $x^*=\la(1-g^w)$ for some $\la\in\mathbb C$.
Since $*$ is an involution and a conjugate-linear antialgebra endomorphism of $H$, we have
$x=(x^*)^*=(\la(1-g^w))^*=\overline{\la}(1-g^{w^2})$. This is impossible.
Hence $w=1$, and so $g^*=g$ and $\t(x^\ast)=x^\ast\ot g+1\ot x^\ast$.
Then by Lemma \ref{4.6}, $x^\ast=\a_{11} x+\a_{12} y+\a_{13}(1-g)$
for some $\a_{11}, \a_{12}, \a_{13}\in\mathbb {C}$.
Similarly, one can show that $y^\ast=\a_{21} x+\a_{22}y+\a_{23}(1-g)$
for some $\a_{21}, \a_{22}, \a_{23}\in\mathbb {C}$.
Then by $xg=\omega gx$, one gets that $(xg)^*=(\omega gx)^*$.
However, $(xg)^*=g^*x^*=g(\a_{11}x+\a_{12}y+\a_{13}(1-g))=\a_{11}gx+\a_{12}gy+\a_{13}(g-g^2)$
and $(\omega gx)^*=\overline{\omega}x^*g^*=\omega^{-1}(\a_{11}x+\a_{12}y+\a_{13}(1-g))g
=\omega^{-1}\a_{11}xg+\omega^{-1}\a_{12}yg+\omega^{-1}\a_{13}(g-g^2)
=\a_{11}gx+\omega^{-2}\a_{12}gy+\omega^{-1}\a_{13}(g-g^2)$.
It follows that $\a_{12}=\omega^{-2}\a_{12}$ and $\a_{13}=\omega^{-1}\a_{13}$.
Hence $\a_{12}(1-\omega^2)=0$ and $\a_{13}=0$ by $\omega\neq 1$.
Similarly, from $(gy)^*=(\omega yg)^*$, one gets that $\a_{21}(1-\omega^2)=0$ and $\a_{23}=0$.

(1) Assume that $n>2$. Then $\omega^2\neq 1$, and hence $\a_{12}=\a_{21}=0$
by $\a_{12}(1-\omega^2)=0$ and $\a_{21}(1-\omega^2)=0$.
Thus, $x^*=\a_{11}x$ and $y^*=\a_{22}y$.
Then we have $x=(x^*)^*=(\a_{11}x)^*
=\overline{\a_{11}}x^*=\overline{\a_{11}}\a_{11}x$,
which implies that $|\a_{11}|=1$. Similarly, one can show that $|\a_{22}|=1$.
This shows Part (1).

(2) Assume that $n=2$. Then $x^*=\a_{11}x+\a_{12}y$ and $y^*=\a_{21}x+\a_{22}y$.
Hence we have $x=(x^*)^*=(\a_{11}x+\a_{12}y)^*=\ol{\a_{11}}x^*+\ol{\a_{12}}y^*
=\ol{\a_{11}}(\a_{11}x+\a_{12}y)+\ol{\a_{12}}(\a_{21}x+\a_{22}y)
=(\ol{\a_{11}}\a_{11}+\ol{\a_{12}}\a_{21})x+(\ol{\a_{11}}\a_{12}+\ol{\a_{12}}\a_{22})y$
and
$y=(y^*)^*=(\a_{21}x+\a_{22}y)^*=\ol{\a_{21}}x^*+\ol{\a_{22}}y^*
=\ol{\a_{21}}(\a_{11}x+\a_{12}y)+\ol{\a_{22}}(\a_{21}x+\a_{22}y)
=(\ol{\a_{21}}\a_{11}+\ol{\a_{22}}\a_{21})x+(\ol{\a_{21}}\a_{12}+\ol{\a_{22}}\a_{22})y$.
It follows that
$$\left(
    \begin{array}{cc}
      \ol{\a_{11}} & \ol{\a_{12}} \\
      \ol{\a_{21}} & \ol{\a_{22}} \\
    \end{array}
  \right)
  \left(
    \begin{array}{cc}
      \a_{11} & \a_{12} \\
      \a_{21} & \a_{22} \\
    \end{array}
  \right)
  =\left(
    \begin{array}{cc}
      1 & 0 \\
      0 & 1 \\
    \end{array}
  \right).$$
This shows Part (2).
\end{proof}

\begin{theorem}\label{4.9}
If $n\geqslant 3$, then up to equivalence, there is a unique Hopf $\ast$-algebra structure on $H$ given by
$$g^*=g,\ x^*=x,\ y^*=y.$$
\end{theorem}

\begin{proof}
Assume that $n\geqslant 3$. Then by Lemma \ref{3.1},
the relations given in the theorem determine a Hopf $*$-algbera structure on $H$, denoted by $*'$.
Now let $*$ be any Hopf $*$-algebra structure on $H$. Then by Lemma \ref{3.1} and
Theorem \ref{4.8}(1), there exist elements $\a, \b \in\mathbb{C}$ with $|\a|=|\b|=1$ such that
$$g^*=g, \ x^*=\a x, \ y^*=\b y.$$
Pick up two elements $\la_1, \la_2\in\mathbb{C}$ with $\la_1^2=\a$ and $\la_2^2=\b$.
Then $|\la_1|=|\la_2|=1$ by $|\a|=|\b|=1$, and hence $\la_1^{-1}=\overline{\la_1}$
and $\la_2^{-1}=\overline{\la_2}$. It is easy to see that there is a Hopf algebra automorphism
$\phi$ of $H$ such that $\phi(g)=g$, $\phi(x)=\la_1x$ and $\phi(y)=\la_2y$.
Then $\phi(g^{*'})=\phi(g)=g=g^*=\phi(g)^*$,
$\phi(x^{*'})=\phi(x)=\la_1x=\la_1^{-1}\a x=\overline{\la_1}x^*=(\la_1x)^*=\phi(x)^*$
and $\phi(y^{*'})=\phi(y)=\la_2y=\la_2^{-1}\b y=\overline{\la_2}y^*=(\la_2y)^*=\phi(y)^*$.
Hence $\phi(h^{*'})=\phi(h)^*$ for all $h\in H$, and so $*$ is equivalent to $*'$.
\end{proof}

Throughout the following, assume that $n=2$. In this case, $\omega=-1$.

Let $A=\left(
         \begin{array}{cc}
           \a_{11} & \a_{12} \\
           \a_{21} & \a_{22} \\
         \end{array}
       \right)$
and $B=\left(
         \begin{array}{cc}
           \b_{11} & \b_{12} \\
           \b_{21} & \b_{22} \\
         \end{array}
       \right)$
be two matrices in $M_2(\mathbb{C})$ with $\ol{A}A=\ol{B}B=I_2$,
and let $*_A$ and $*_B$ be the corresponding Hopf $*$-algebra structures on $H$
determined by $A$ and $B$ as in Lemma \ref{3.2}, respectively.
Then we have the following proposition.

\begin{proposition}\label{4.10}
$*_A$ and $*_B$ are equivalent $*$-structures on $H$ if and only if there exists
an invertible matrix
$\La$ in $M_2(\mathbb{C})$ such that $A\La=\ol{\La}B$, i.e.,  $\ol{\La}^{-1}A\La=B$.
\end{proposition}

\begin{proof}
Suppose that $*_A$ and $*_B$ are equivalent. Then there exists a Hopf algebra automorphism
$\phi$ of $H$ such that $\phi(h^{*_A})=\phi(h)^{*_B}$ for all $h\in H$.
By Lemma \ref{3.4} and $n=2$, one can see that $\phi(g)=g$.
Then by Lemma \ref{4.6}, a straightforward computation shows that there exists a matrix
$\La=\left(
         \begin{array}{cc}
           \la_{11} & \la_{12} \\
           \la_{21} & \la_{22} \\
         \end{array}
       \right)$ in $M_2(\mathbb{C})$
such that $\phi(x)=\la_{11}x+\la_{12}y$
and $\phi(y)=\la_{21}x+\la_{22}y$.
Since $\phi$ is an isomorphism, one can check that $\La$ is an invertible matrix in $M_2(\mathbb C)$.
Now we have
$$\begin{array}{rl}
\phi(x^{*_A})&=\phi(\a_{11}x+\a_{12}y)=\a_{11}\phi(x)+\a_{12}\phi(y)\\
&=\a_{11}(\la_{11}x+\la_{12}y)+\a_{12}(\la_{21}x+\la_{22}y)\\
&=(\a_{11}\la_{11}+\a_{12}\la_{21})x+(\a_{11}\la_{12}+\a_{12}\la_{22})y\\
\end{array}$$
and
$$\begin{array}{rl}
\phi(x)^{*_B}&=(\la_{11}x+\la_{12}y)^{*_B}=\ol{\la_{11}}x^{*_B}+\ol{\la_{12}}y^{*_B}\\
&=\ol{\la_{11}}(\b_{11}x+\b_{12}y)+\ol{\la_{12}}(\b_{21}x+\b_{22}y)\\
&=(\ol{\la_{11}}\b_{11}+\ol{\la_{12}}\b_{21})x+(\ol{\la_{11}}\b_{12}+\ol{\la_{12}}\b_{22})y.\\
\end{array}$$
Hence it follows from $\phi(x^{*_A})=\phi(x)^{*_B}$ that
$\a_{11}\la_{11}+\a_{12}\la_{21}=\ol{\la_{11}}\b_{11}+\ol{\la_{12}}\b_{21}$
and $\a_{11}\la_{12}+\a_{12}\la_{22}=\ol{\la_{11}}\b_{12}+\ol{\la_{12}}\b_{22}$.
Similarly, from $\phi(y^{*_A})=\phi(y)^{*_B}$, one gets that
$\a_{21}\la_{11}+\a_{22}\la_{21}=\ol{\la_{21}}\b_{11}+\ol{\la_{22}}\b_{21}$
and $\a_{21}\la_{12}+\a_{22}\la_{22}=\ol{\la_{21}}\b_{12}+\ol{\la_{22}}\b_{22}$.
Thus, we have $A\La=\ol{\La}B$.

Conversely, suppose that there exists an invertible matrix
$\La=\left(
         \begin{array}{cc}
           \la_{11} & \la_{12} \\
           \la_{21} & \la_{22} \\
         \end{array}
       \right)$ in $M_2(\mathbb{C})$ such that $A\La=\ol{\La}B$.
Then it is straightforward to check that there is a Hopf algebra automorphism $\phi$ of $H$
uniquely determined by $\phi(g)=g$,  $\phi(x)=\la_{11}x+\la_{12}y$ and $\phi(y)=\la_{21}x+\la_{22}y$.
Obviously, $\phi(g^{*_A})=\phi(g)^{*_B}=g$.
From the computation above, one gets that $\phi(x^{*_A})=\phi(x)^{*_B}$ and $\phi(y^{*_A})=\phi(y)^{*_B}$.
It follows that $\phi(h^{*_A})=\phi(h)^{*_B}$ for any $h\in H$.
This shows that $*_A$ and $*_B$ are equivalent.
\end{proof}

{ \bf Acknowledgements}
This work supported by the National Natural Science Foundation of China (Grant No.11571298, 11711530703).

\end{document}